\documentclass[reqno,a4paper,11pt]{amsart}

\usepackage[all,poly]{xy}
\usepackage{amsfonts}
\usepackage[mathcal]{eucal}
\usepackage{amssymb}
\usepackage{amsmath}
\usepackage{mathrsfs}
\usepackage{color}
\usepackage[pagebackref,colorlinks]{hyperref}
\usepackage{enumerate}

\theoremstyle{plain}
\newtheorem {lemma}{Lemma} 
\newtheorem {theorem}[lemma]{Theorem}

\newtheorem {corollary}[lemma]{Corollary}

\newtheorem {proposition}[lemma]{Proposition}

\theoremstyle{definition}
\newtheorem {remark}[lemma]{Remark}

\newtheorem {example}[lemma]{Example}

\theoremstyle{definition}

\newtheorem{deff}[lemma]{Definition}{}

\newcommand{\op}{\operatorname{op}}
\newcommand{\M}{\operatorname{\mathbb M}}
\newcommand{\LL}{\operatorname{\mathcal L}}

\newcommand{\gr}{\operatorname{gr}}

\newcommand{\BF}{\operatorname{BF}}

\newcommand{\ga}{\gamma}

\newcommand{\Z}{\mathbb Z}
\newcommand{\N}{\mathbb N}

\newcommand{\Fdim}{\operatorname{Fdim-}}
\newcommand{\fdim}{\operatorname{fdim-}}

\newcommand{\PProj}{\operatorname{Proj}}

\newcommand{\Qcoh}{\operatorname{QCoh}}
\newcommand{\QGr}{\operatorname{QGr-}}
\newcommand{\qgr}{\operatorname{qgr-}}

\newcommand{\End}{\operatorname{End}}

\newcommand{\Hom}{\operatorname{Hom}}

\newcommand{\id}{\operatorname{id}}

\newcommand\Gr[1][]{{\operatorname{{Gr}^{#1}-}}}

\newcommand{\grr}{\operatorname{gr-}}

\newcommand{\Modd}{\operatorname{Mod-}}

\title{The dynamics of Leavitt path algebras}

\author{R.~Hazrat}\address{
School of Computing, Engineering and Mathematics\\
University of Western Sydney\\
Australia}

\email{r.hazrat@uws.edu.au}


\begin{document}
\begin{abstract} 
Recently it was shown that the notion of flow equivalence of shifts of finite type in symbolic dynamics is related to the Morita theory and the Grothendieck group in the theory of Leavitt path algebras~\cite{flowa}. In this paper we show that the notion of conjugacy of shifts of finite type is closely related to the {\it graded} Morita theory and consequently the {\it graded} Grothendieck group. This fits into the general 
framework we have in these two theories: Conjugacy yields the flow equivalence and the graded Morita equivalence can be lifted to the Morita equivalence. 
Starting from a finite directed graph, the observation that the graded Grothendieck group of the Leavitt path algebra associated to $E$ coincides with the Krieger dimension group of the shift of finite type  associated to $E$ provides a link between the theory of Leavitt path algebras and symbolic dynamics.  It has been conjectured that the ordered graded Grothendieck group as $\mathbb Z[x,x^{-1}]$-module (we call this the graded dimension group) classifies the Leavitt path algebras completely~\cite{hazann}.  Via the above correspondence, utilising the results from symbolic dynamics, we prove that for two purely infinite simple unital Leavitt path algebras, if their graded dimension groups are isomorphic, then the algebras are isomorphic.
\end{abstract}

\maketitle

\section{Introduction}\label{intro}

One of the central objects in the theory of symbolic dynamics is a shift of finite type (i.e., a topological Markov chain). 
Every finite directed graph $E$ with no sinks and sources gives rise to a shift of finite type $X_E$ by considering the set of bi-infinite paths and the natural shift of the paths to the left. This is called an edge shift. Conversely any shift of finite type is conjugate to an edge shift. 

There are two notions of equivalences in the classifications of shifts of finite type: the conjugacy and the weaker notion of flow equivalence. 
Two shifts of finite type $X_E$ and $X_F$ are conjugate if $E$ can be obtained from $F$ by a series of in/out-splitting (see~\cite[Theorem~7.1.2, Corollary~7.1.5]{lindmarcus}). 
Furthermore, $X_E$ and $X_F$ are flow equivalent if $E$ can be obtained from $F$ by a series of in/out-splitting and expansion of graphs (see~\cite[456]{lindmarcus} and~\cite{parrysullivan}). 
 
On the other hand, to a directed graph one can associate an analytical object, called a graph $C^*$-algebra~\cite{Raegraph} and an algebraic object called, a Leavitt path algebras~\cite{aap05,amp}.

The relation between symbolic dynamics and (graph) $C^*$-algebras were explored in the work of Cuntz and Krieger~\cite{cuntzkrieger} and later Bates and Pask~\cite{batespask}, among others. The algebraic counter part of graph $C^*$-algebras, i.e., Leavitt path algebras, are related to symbolic dynamics by the recent work of Abrams, Louly, Pardo and Smith~\cite{flowa}. In both settings, it was observed that the notion of flow equivalence is related to the theory of Morita equivalence and the Grothendieck group $K_0$. It was shown that for two essential finite graphs $E$ and $F$, if $X_E$ is flow equivalent to $X_F$, then $C^*(E)$ is Morita equivalent to $C^*(F)$ and $\LL(E)$ is Morita equivalent to $\LL(F)$ (see~\cite{batespask} and~\cite{flowa}, respectively). 

Inspired by this connection between symbolic dynamics and graph algebras, in this note we show that, whereas, the flow equivalence is related to the Morita theory and the Grothendieck group $K_0$,  the notion of conjugacy in symbolic dynamics is closely related to the {\it graded} Morita theory and consequently the {\it graded} Grothendieck group $K^{\gr}_0$ in the setting of Leavitt path algebras. This fits into the general framework we have in these two theories: Conjugacy yields the flow equivalence and the graded Morita equivalence can be lifted to the Morita equivalence (see~\S\ref{hgdeiii}). 

Thanks to the deep work of Williams~\cite{williams}, there is a matrix criterion when two edge shifts are conjugate; $X_E$ and $X_F$  are conjugate if and only if the adjacency matrices of $E$ and $F$ are strongly shift equivalent (see~\S\ref{willi}). Williams further introduced a weaker notion of the shift equivalence. Inspired by the success of $K$-theory in classification of AF $C^*$-algebras~\cite{elliot}, Krieger introduced a variation of $K_0$ which he showed is a complete invariant for the shift equivalence~\cite{krieger1}.  It can be observed that Krieger's dimension group (and Wagoner's dimension module~\cite{wago1,wago2}) coincides with the graded Grothendieck group of associated Leavitt path algebras. This provides a bridge from the theory of Leavitt path algebras to symbolic dynamics. 

Throughout this paper, we only will consider finite graphs with no sinks (and no sources in several occasions). The reason is, considering the edge shift $X_E$ of a graph $E$, no arrow that begins at a source and no arrow that ends at a sink appear in any bi-infinite path (see the paragraph before the definition of essential graphs~\cite[Definition~2.2.9]{lindmarcus}). Thus only graphs with no sinks and sources appear in symbolic dynamics. Furthermore, one of the most interesting classes of Leavitt path algebras, i.e., purely infinite simple unital algebras are within this class of graphs.

The results of the paper are summarised as follows. Let $E$ and $F$ be two finite directed graphs with no sinks and sources and $A_E$ and $A_F$ be their adjacency matrices, respectively. 

\begin{description}

\item[Corollary~\ref{h99}] The matrices $A_E$ and $A_F$ are shift equivalent if and only if there is an order preserving $\mathbb Z[x,x^{-1}]$-module  isomorphism
$K_0^{\gr}(\LL(E)) \cong_{\gr} K_0^{\gr}(\LL(F))$.

\medskip

\item[Proposition~\ref{hgysweet}] If $A_E$ and $A_F$  are strongly shift equivalent then $\LL(E)$ and $\LL(F)$ are graded Morita equivalent. Conversely, if $\LL(E)$ and $\LL(F)$ are graded Morita equivalent, then  $A_E$ and $A_F$  are shift equivalent.

\item[Example~\ref{hhyyuu}] The strongly shift equivalence does not imply isomorphisms between Leavitt path algebras.

\medskip

\item[Theorem~\ref{thucomingtobal}] Suppose $\LL(E)$ and $\LL(F)$ are purely infinite simple unital algebras. Then $\LL(E) \cong \LL(F)$ if  there is an order preserving $\mathbb Z[x,x^{-1}]$-module isomorphism
\begin{equation*}
\big (K_0^{\gr}(\LL(E)),[\LL(E)]\big ) \cong \big (K_0^{\gr}(\LL(F)),[\LL(F)]\big ).
\end{equation*}

\medskip

For a graph $E$, let $\mathcal P(E)$ be the associated path algebra and $\Gr \mathcal P(E)$ be the  category of $\mathbb Z$-graded right $\mathcal P(E)$-modules and  $\Fdim \mathcal P(E)$ be its full (Serre) subcategory of modules that are the sum of their finite-dimensional submodules. Paul Smith~\cite{smith1,smith2,smith3} recently studied the 
quotient category
\[\QGr \mathcal P(E) := \Gr \mathcal P(E) / \Fdim \mathcal P(E).\]

\item[Theorem~\ref{algnihy}] There is an ordered abelian group isomorphism  $K^{\gr}_0(\LL(E)) \cong K^{\gr}_0(\LL(F))$  if and only if 
$\QGr \mathcal P(E)  \approx \QGr \mathcal P(F)$.
\end{description}

The main aim of this paper is to provide evidence that the graded Grothendieck group is a capable invariant which could eventually provide a bridge between the theories of graph $C^*$-algebras and Leavitt path algebras via symbolic dynamics. This group comes equipped with an pre-ordered abelian group structure, plus an action of $\mathbb Z$ on it, (which corresponds to the shifting in a shift space), which makes it a $\mathbb Z[x,x^{-1}]$-module, and a distinguished element, namely the identity (see~\S\ref{gftegtds}). In~\S\ref{oo12} we show that there is a product formula for this $K$-group. Section~\ref{willi} relates it to conjugacy in symbolic dynamics. And in~\S\ref{noncomui} we show that only the pre-order part of this group is sufficient to classify the the quotient category of path algebras. 

The paper is organised as follows. In~\S\ref{hgdeiii} we recall the notion of graded Morita theory. 
For an arbitrary group $\Gamma$, one can equip the Leavitt path algebra $\LL(E)$ associated to the graph $E$ with a $\Gamma$-graded structure. This is recalled in~\S\ref{poidget}. The graded Grothendieck group as an invariant for classification of Leavitt path algebras was first considered in~\cite{hazann}. In~\S\ref{gftegtds} we recall this group. In fact the graded Grothendieck group is not only an ordered abelian group, but has $\Z[x,x^{-1}]$-modules structure. The action of $x$ on the group captures the shifting in the corresponding shift of finite type. In~\S\ref{jjhh1} we observe that for a finite directed graph, the graded Grothendieck group of the Leavitt path algebra associated to $E$ coincides with the Krieger dimension group of the shift of finite type associated to $E$. This provides a link between the theory of Leavitt path algebras and symbolic dynamics. This has also been recently observed by Ara and Pardo, using a different approach in~\cite{arapardo}.  In Section~\ref{oo12} we study the behaviour of the graded Grothendieck group on the product of the graphs. We will establish a formula to express $K^{\gr}_0$ of the product of the graphs as the tensor product of $K^{\gr}_0$ of the graphs. 
Section~\ref{willi} is the main part of the paper, where the relations between strongly shift equivalence and graded Morita equivalence are studied. 
Section~\ref{noncomui} shows that the graded Grothendieck group is a complete invariant for the quotient category of  path algebras. In fact, here, we only need the structure of ordered abelian group (not the module structure) of the graded Grothendieck group to classify the quotient categories.

\section{Graded Morita theory}\label{hgdeiii}

For a graded ring $A$, the graded Grothendieck group $K^{\gr}_0(A)$ is constructed from the category of graded finitely generated projective $A$-modules. Thus it is natural to consider categories of graded modules which are equivalent, which in turn induces isomorphic graded Grothendieck groups. We will see that graded equivalence of categories are closely related to the notion of shift equivalence in symbolic dynamics in~\S\ref{willi}. 

In this section we gather results on graded Morita theory that we need in the paper. For the theory of graded rings, we refer the reader to~\cite{grrings}. 

For an abelian group $\Gamma$ and a $\Gamma$-graded ring $A$, by $\Gr A$, we denote a category consists of graded right $A$-modules as objects and  graded homomorphisms as the morphisms.  For $\alpha \in \Gamma$,  the {\it $\alpha$-suspension functor} or {\it shift functor} $\mathcal T_\alpha:\Gr A\rightarrow \Gr A$, $M \mapsto M(\alpha)$  is an isomorphism with the property
$\mathcal T_\alpha \mathcal T_\beta=\mathcal T_{\alpha + \beta}$, $\alpha,\beta\in \Gamma$.

\begin{deff} \label{grdeffsa} 
Let $A$ and $B$ be $\Gamma$-graded rings. 

\begin{enumerate}

\item A functor $\phi:\Gr A \rightarrow \Gr B$ is called a \emph{graded functor} if $\phi \mathcal T_{\alpha} = \mathcal T_{\alpha} \phi$. 

\medskip 

\item A graded functor $\phi:\Gr A \rightarrow \Gr B$ is called a \emph{graded equivalence} if there is a graded functor $\psi:\Gr B \rightarrow \Gr A$ such that $\psi \phi \cong_{\gr} 1_{\Gr A}$ and $\phi \psi \cong_{\gr} 1_{\Gr B}$. 

\medskip

\item If there is a graded equivalence between $\Gr A$ and $\Gr B$, we say $A$ and $B$ are {\it graded equivalent} or {\it graded Morita equivalent} and we write 
$\Gr A \approx_{\gr} \Gr B$, or $\Gr[\Gamma] A \approx_{\gr} \Gr[\Gamma] B$ to emphasis the categories are $\Gamma$-graded.

\medskip 

\item A functor $\phi: \Modd A \rightarrow \Modd B$ is called a \emph{graded functor} if  there is a graded functor $\phi': \Gr A \rightarrow \Gr B$ such that the following diagram, where the vertical functors are forgetful functors, commutes.
\begin{equation}\label{njhysi}
\xymatrix{
\Gr A \ar[r]^{\phi'} \ar[d]_{F}& \Gr B \ar[d]^{F}\\
\Modd A \ar[r]^{\phi}  & \Modd B.
}
\end{equation}

The functor $\phi'$ is called an \emph{associated graded functor} of $\phi$. 
\medskip 

\item A functor $\phi:\Modd A \rightarrow \Modd B$ is called a \emph{graded equivalence} if it is graded and an equivalence. 

\end{enumerate}
\end{deff}

For a ring $A$, and a full idempotent element $e\in A$ (i.e., $e^2=e$ and $AeA=A$), it is well-know that the ring $A$ is Morita equivalent to $eAe$. In Example~\ref{idempogr} we establish a similar statement in the graded setting which will be used in Proposition~\ref{valenjov} and Theorem~\ref{cafejen1}. 

\begin{example}\label{idempogr}
Let $A$ be a graded ring and $e$ be a full homogeneous idempotent of $A$, i.e., $e^2=e$ and $AeA=A$. Clearly $e$ has  degree zero. 
Consider $P=eA$. One can readily see that $P$ is a right graded progenerator. Then $P^*=\Hom_A(eA,A)\cong_{\gr}Ae$ and $B=\End_A(eA,eA)\cong_{\gr} eAe$. 
The $A-A$-bimodule graded homomorphism $\phi:Ae\otimes_{eAe}eA \rightarrow A$ and the $eAe-eAe$-bimodule graded homomorphism  $\psi:eA\otimes_A Ae \rightarrow eAe$ are isomorphic. 
Consequently one can check that the functors $-\otimes_A Ae: \Gr A\rightarrow \Gr eAe$ and $-\otimes_{eAe} eA: \Gr eAe \rightarrow \Gr A$ are inverse of each other.
 Thus we get an (graded) equivalence between $\Gr A$ and $\Gr eAe$ which lifts to an (graded) equivalence between $\Modd A$ and $\Modd eAe$, as it is shown in the diagram below.
\begin{equation*}
\xymatrix{
\Gr A \ar[rr]^{-\otimes_A Ae} \ar[d]_{F}&& \Gr eAe \ar[d]^{F}\\
\Modd A \ar[rr]^{-\otimes_A Ae}  && \Modd eAe.
}
\end{equation*}
\end{example}

Example~\ref{idempogr} shows that a graded equivalence between the categories $\Gr A$ and $\Gr eAe$ can be lifted to an equivalence between the categories $\Modd A$ and $\Modd eAe$. This lifting of the graded equivalence is a general phenomena as proved by Gordon and Green~\cite[Proposition~5.3, Theorem~5.4]{greengordon}. Since we need this result in several occasions, we record it here. 

\begin{theorem}\label{grmorim11} 
Let $A$ and $B$ be two $\Gamma$-graded rings. The following are equivalent:

\begin{enumerate}[\upshape(1)]

\item $\Modd A$ is graded equivalent to $\Modd B$;

\item $\Gr A$ is graded equivalent to $\Gr B$;

\item $B\cong_{\gr} \End_A(P)$ for a graded $A$-progenerator $P$;

\item $B\cong_{\gr} e \M_n(A)(\overline \delta) e$ for a full homogenous idempotent $e \in \M_n(A)(\overline \delta)$, where $\overline \delta=(\delta_1,\dots,\delta_n)$, $\delta_i \in \Gamma$.  
\end{enumerate}
 \end{theorem}



\section{Grading on Leavitt path algebras}\label{poidget}

For an arbitrary group $\Gamma$, one can equip $\LL(E)$ with a $\Gamma$-graded structure. This will be needed in the note (see the proof of Theorem~\ref{cafejen1}). 
We first recall the definition of a Leavitt path algebra associated to a graph.

\begin{deff}\label{llkas}{\sc Leavitt path algebras.} \label{LPA} \\
For a row-finite graph $E$ and a ring $R$ with identity, the {\it Leavitt path algebra of $E$}, denoted by $\LL_R(E)$, is the algebra generated by the sets $\{v \mid v \in E^0\}$, $\{ \alpha \mid \alpha \in E^1 \}$ and $\{ \alpha^* \mid \alpha \in E^1 \}$ with the coefficients in $R$, subject to the relations 

\begin{enumerate}
\item $v_iv_j=\delta_{ij}v_i \textrm{ for every } v_i,v_j \in E^0$.

\item $s(\alpha)\alpha=\alpha r(\alpha)=\alpha \textrm{ and }
r(\alpha)\alpha^*=\alpha^*s(\alpha)=\alpha^*  \textrm{ for all } \alpha \in E^1$.

\item $\alpha^* \alpha'=\delta_{\alpha \alpha'}r(\alpha)$, for all $\alpha, \alpha' \in E^1$.

\item $\sum_{\{\alpha \in E^1, s( \alpha)=v\}} \alpha \alpha^*=v$ for every $v\in E^0$ for which $s^{-1}(v)$ is non-empty.

\end{enumerate}
\end{deff}
Here the ring $R$ commutes with the generators $\{v,\alpha, \alpha^* \mid v \in E^0,\alpha \in E^1\}$. Throughout this note the coefficient ring is a fixed field $K$ and we simply write $\LL(E)$ instead of $\LL_K(E)$. The elements $\alpha^*$ for $\alpha \in E^1$ are called {\it ghost edges}. One can show that $\LL(E)$ is a ring with identity if and only if the graph $E$ is finite (otherwise, $\LL(E)$ is a ring with local identities).

Let $\Gamma$ be an arbitrary group with the identity element $e$. Let $w:E^1\rightarrow \Gamma$ be a {\it weight} map and further define $w(\alpha^*)=w(\alpha)^{-1}$, for $\alpha \in E^1$ and $w(v)=e$ for $v\in E^0$.  The free $K$-algebra generated by the vertices, edges and ghost edges is a $\Gamma$-graded $K$-algebra. Furthermore, the Leavitt path algebra is the quotient of this algebra by relations in  Definition~\ref{LPA} which are all homogeneous. Thus $\LL_K(E)$ is a $\Gamma$-graded $K$-algebra. 

\begin{example}
Consider the graphs 
 
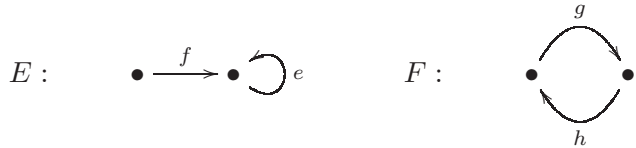
\begin{equation*}
\xymatrix{
E: & \bullet \ar[r]^f &  \bullet \ar@(rd,ru)_{e} & & F:  &  \bullet \ar@/^1.5pc/[r]^g & \bullet \ar@/^1.5pc/[l]^h & \\
}
\end{equation*}
Assigning $0$ to vertices and $1$ to edges in the graphs in the usual manner, by~\cite[Theorem~4.2]{haz} we obtain
$\LL(E)\cong_{\gr}\M_2(K[x,x^{-1}])(0,1)$ whereas $\LL(F)\cong_{\gr}\M_2(K[x^2,x^{-2}])(0,1)$ and one can easily observe that $\LL_K(E)\not \cong_{\gr} \LL_K(F)$.

However assigning $1$ for the degree of  $f$ and $2$ for the degree of $e$ in $E$ and $1$ for the degrees of $g$ and $h$ in $F$, the proof of~\cite[Theorem~4.2]{haz} shows that
$\LL_K(E)\cong \M_2(K[x^2,x^{-2}])(0,1)$ and $\LL_K(F)\cong \M_2(K[x^2,x^{-2}])(0,1)$. So with these gradings, $\LL_K(E)\cong_{\gr} \LL_K(F)$. 
\end{example}

The natural and standard grading given to a Leavitt path algebra is a $\mathbb Z$-grading by setting $\deg(v)=0$, for $v\in E^0$, $\deg(\alpha)=1$ and $\deg(\alpha^*)=-1$ for $\alpha \in E^1$. 
If $\mu=\mu_1\dots\mu_k$, where $\mu_i \in E^1$, is an element of $\LL(E)$, then we denote by $\mu^*$ the element $\mu_k ^*\dots \mu_1^* \in \LL(E)$. Further we define $v^*=v$ for any $v\in E^0$. Since $\alpha^* \alpha'=\delta_{\alpha \alpha'}r(\alpha)$, for all $\alpha, \alpha' \in E^1$, any word in the generators $\{v, \alpha, \alpha^* \mid v\in E^0, \alpha \in E^1   \}$ in $\LL(E)$ can be written as $\mu \gamma ^*$ where $\mu$ and $\gamma$ are paths in $E$ (vertices are considered paths of length zero).  The elements of the form $\mu\gamma^*$ are called {\it monomials}. 

Taking the grading into account, one can write $\LL(E) =\textstyle{\bigoplus_{k \in \mathbb Z}} \LL(E)_k$ where,
\[\LL(E)_k=  \Big \{ \sum_i r_i \alpha_i \beta_i^*\mid \alpha_i,\beta_i \textrm{ are paths}, r_i \in K, \textrm{ and } |\alpha_i|-|\beta_i|=k \textrm{ for all } i \Big\}.\] 
The following theorem was proved in~\cite{haz} which determined finite graphs whose associated Leavitt path algebras are strongly $\mathbb Z$-graded (see also~\cite{haziso} for another proof by realising Leavitt path algebras as corner skew Laurent polynomial rings).

\begin{theorem}\label{sthfin}
Let $E$ be a finite graph.  Then $L(E)$ is strongly graded if and only if  $E$ does not have sinks.
\end{theorem}

This theorem along with Dade's theorem (see~\S\ref{volkhtes5}) will be used throughout this paper to pass from the graded $K$-theory to the non-graded $K$-theory of ring of homogeneous elements of degree zero.

\section{Graded Grothendieck groups and graded dimension groups}\label{gftegtds}

For an abelian monoid $V$, we denote by $V^+$ the group completion of $V$. This gives a left adjoint functor to the forgetful functor from the category of abelian groups to abelian monoids. 
When the monoid $V$ has a $\Gamma$-module structure, where $\Gamma$ is a group, then $V^+$ inherits a natural $\Gamma$-module structure, or equivalently,  $\Z[\Gamma]$-module structure. 

The graded Grothendieck group of a graded ring is constructed as the completion of the abelian monoid of isomorphic classes of graded finitely generated projective modules.  Namely,  for a $\Gamma$-graded ring $A$ and a graded finitely generated projective (right) $A$-module $P$, let $[P]$ denote the class of graded finitely generated projective modules graded isomorphic to $P$. Then the monoid 
\begin{equation}\label{zhongshan1}
\mathcal V^{\gr}(A)=\big \{[P] \mid  P  \text{ is graded finitely generated projective A-module} \big \}
\end{equation}
has a $\Gamma$-module structure defined as follows: for $\ga \in \Gamma$ and $[P]\in \mathcal V^{\gr}(A)$, $\ga .[P]=[P(\ga)]$. The group $\mathcal V^{\gr}(A)^+$ is called the {\it graded Grothendieck group} and is denoted by $K^{\gr}_0(A)$, which 
as  the above discussion shows is a $\Z[\Gamma]$-module. This extra $\Z[\Gamma]$-module carries a substantial information about the graded ring $A$. 

The main aim of this note is to concentrate on the graded Grothendieck group of Leavitt path algebras as a possible invariant for these algebras. This line of study started in~\cite{hazann}. In the case of finite graphs with no sinks, there is a good description of the action of the group on the graded Grothendieck group which we recall here.

Let $E$ be a finite graph with no sinks. Set $\mathcal  A=\LL(E)$ which is a strongly $\mathbb Z$-graded ring by Theorem~\ref{sthfin}.  For any $u \in E^0$ and $i \in \mathbb Z$, $u\mathcal  A(i)$ is a right graded finitely generated projective $\mathcal  A$-module and any graded finitely generated projective $\mathcal  A$-module is generated by these modules up to isomorphism, i.e., 
$$\mathcal V^{\gr}(\mathcal  A)=\Big \langle \big [u\mathcal  A(i)\big ]  \mid u \in E^0, i \in \mathbb Z \Big \rangle, 
$$ and $K_0^{\gr}(\mathcal  A)$ is the group completion of $\mathcal V^{\gr}(\mathcal  A)$. The action of $\mathbb N[x,x^{-1}]$ on $\mathcal V^{\gr}(\mathcal A)$ and thus the action of $\mathbb Z[x,x^{-1}]$ on $K_0^{\gr}(\mathcal  A)$ is defined on generators by $x^j [u\mathcal  A(i)]=[u\mathcal  A(i+j)]$, where $i,j \in \mathbb Z$. We first observe that for $i\geq 0$, 
\begin{equation}\label{hterw}
x[u\mathcal  A(i)]=[u\mathcal  A(i+1)]=\sum_{\{\alpha \in E^1 \mid s(\alpha)=u\}}[r(\alpha)\mathcal  A(i)].
\end{equation}

First notice that for $i\geq 0$, $\mathcal  A_{i+1}=\sum_{\alpha \in E^1} \alpha \mathcal  A_i$. It follows \[u \mathcal  A_{i+1}=\bigoplus_{\{\alpha \in E^1 \mid s(\alpha)=u\}} \alpha \mathcal  A_i\] as $\mathcal  A_0$-modules. Using the fact that $\mathcal  A_n\otimes_{\mathcal  A_0}\mathcal  A\cong \mathcal  A(n)$, $n \in \mathbb Z$, and the fact that $\alpha \mathcal  A_i \cong r(\alpha) \mathcal  A_i$ as $\mathcal  A_0$-module,
we get \[u\mathcal  A(i+1) \cong \bigoplus_{\{\alpha \in E^1 \mid s(\alpha)=u\}} r(\alpha) \mathcal  A(i)\] as graded $\mathcal  A$-modules. This gives~(\ref{hterw}).

\subsection{} \label{conjisi}
Let $V$ be an abelian monoid, $\Gamma$ a group and $V$ be a (left) $\Gamma$-module. Let $\geq$ be a reflexive and transitive relation on $V$ which respects the monoid and the module structures, i.e., for $\gamma \in \Gamma$ and $x,y,z \in V$, if $x\geq y$, then $x+z\geq y+z$ and $\gamma x \geq \gamma y$. We call $V$ a {\it $\Gamma$-pre-ordered module}. We call $V$ a {\it pre-ordered module} when $\Gamma$ is clear from the context. The {\it cone} of $V$ is defined as $\{x \in V \mid x\geq 0\}$ and denoted by $V_{+}$.  The set $V_{+}$ is a $\Gamma$-submonoid of $V$, i.e., a submonoid which is closed under the action of $\Gamma$.  In fact, $V$ is a $\Gamma$-pre-ordered module if and only if there exists a $\Gamma$-submonoid of $V$. (Since $V$ is a $\Gamma$-module, it can be  considered as a $\mathbb Z[\Gamma]$-module.)  An element $u \in V_{+}$ is called an {\it order-unit} if for any $x\in V$, there are $\alpha_1,\dots,\alpha_n \in \Gamma$, $n \in \mathbb N$, such that $\sum_{i=1}^n \alpha_i  u \geq x$.  As usual, in this setting, we only consider homomorphisms which preserve the pre-ordering, i.e., a $\Gamma$-homomorphism $f:V\rightarrow W$, such that $f(V_{+}) \subseteq W_{+}$.

For a $\Gamma$-graded ring $A$, $K^{\gr}_0(A)$ is a pre-ordered abelian group with  the set of isomorphic classes of  graded finitely generated  projective right $A$-modules  as the cone of ordering, denoted by $K^{\gr}_0(A)_+$  (i.e., the image of $\mathcal V^{\gr}(A)$ under the natural homomorphism $\mathcal V^{\gr}(A)\rightarrow K^{\gr}_0(A)$). Furthermore, $[A]$ is an order-unit. We call the triple, 
$(K^{\gr}_0(A),K^{\gr}_0(A)_+,[A])$ the {\it graded dimension group} (see~\S\ref{jjhh1} for some background on dimension groups). 

In~\cite{hazann} it was conjectured that the graded dimension group is a complete invariant for Leavitt path algebras. Namely,  for graphs  $E$ and $F$,  $\LL(E)\cong_{\gr} \LL(F)$ if and only if 
there is an order preserving $\mathbb Z[x,x^{-1}]$-module  isomorphism 
\begin{equation}\label{ookjh2}
\phi: K_0^{\gr}(\LL(E)) \rightarrow  K_0^{\gr}(\LL(F))
\end{equation}
such that $\phi([\LL(E)]=\LL(F)$. 
 We denote such an order preserving $\mathbb Z[x,x^{-1}]$-module isomorphism which preserve the position of identity by 
 \begin{equation*}
\big (K_0^{\gr}(\LL(E)),[\LL(E)]\big ) \cong \big (K_0^{\gr}(\LL(F)),[\LL(F)]\big ).
\end{equation*}
It was shown in~\cite{hazann} that the conjecture is valid for the so called polycephaly graphs, which include, acyclic, comets and multi-headed graphs.

\subsection{}\label{volkhtes5}

Let $A$ be a strongly $\Gamma$-graded ring.  By Dade's Theorem(see~\cite[Theorem~2.8]{dade} and ~\cite[Thm.~3.1.1]{grrings}), the functor $(-)_0:\Gr A\rightarrow \Modd A_0$, $M\mapsto M_0$, is an additive functor with an  inverse $-\otimes_{A_0} A: \Modd A_0 \rightarrow \Gr A$ so that it induces an equivalence of categories.  
 This implies
that $K_0^{\gr}(A)  \cong K_0(A_0)$, where $A_0$ is the ring of homogeneous elements of degree zero. Furthermore, under this isomorphism, $K^{\gr}_0(A)_+$ maps onto $ K_0(A_0)_+$ and $[A]$ to $[A_0]$. 

\subsection{}\label{volkhtes}

For the Leavitt path algebra $\LL(E)$, the structure of the ring of homogeneous elements of degree zero, $\LL(E)_0$, is  known and can be represented by a stationary Bratteli diagram. Ordering the vertices of the finite graph $E$, say, $\{u_1,u_2,\dots,u_n\}$, then there are $A_E(i,j)$-lines connecting $u_i$ from the one row of the Bratteli diagram to the $u_j$ of the next row. Here $A_E$ is the adjacency matrix of $E$. Since we need to calculate $K_0(\LL(E)_0)$, we recall the description of $\LL(E)_0$ in the setting of finite graphs with no sinks (see the proof of
Theorem~5.3 in~\cite{amp} which is inspired by~\cite[Proposition~2.3]{cuntzkrieger} in the setting of graph $C^*$-algebras).
Let $L_{0,n}$ be the
linear span of all elements of the form $pq^*$ with $r(p)=r(q)$ and
$|p|=|q|\leq n$. Then 
\begin{equation}\label{ppooii}
\LL(E)_0=\bigcup_{n=0}^{\infty}L_{0,n},
\end{equation} where
the transition inclusion $L_{0,n}\rightarrow L_{0,n+1}$ is
to
identify $pq^*$ with $r(p)=v$ by
\[\sum_{\{ \alpha | s(\alpha)=v\}} p\alpha (q\alpha)^*.\] Note that since $E$ does not have sinks, 
for any $v\in E_0$ the set $\{ \alpha | s(\alpha)=v\}$ is not
empty.

For a fixed $v \in E^0$, let
$L_{0,n}^v$ be the linear span of all elements of
the form $pq^*$ with $|p|=|q|=n$ and $r(p)=r(q)=v$. Arrange the paths
of length $n$ with the range $v$ in a fixed order
$p_1^v,p_2^v,\dots,p_{k^v_n}^v$, and observe that the correspondence
of  $p_i^v{p_j^v}^*$ to the matrix unit $e_{ij}$ gives rise to a ring
isomorphism $L_{0,n}^v\cong\M_{k^v_n}(K)$. Furthermore, $L_{0,n}^v$, $v\in E^0$ form a direct sum.  This implies that
\[L_{0,n}\cong
\bigoplus_{v\in E^0}\M_{k_n^v}(K),\] where $k_n^v$, $v \in E^0$, 
 is the number of paths of length $n$ with the range
$v$. The inclusion map $L_{0,n}\rightarrow L_{0,n+1}$  is
\begin{equation}\label{volleyb}
A_E^t: \bigoplus_{v\in E^0} \M_{k^v_n}(K) \longrightarrow \bigoplus_{v\in E^0}
\M_{k^v_{n+1}}(K).
\end{equation}
  This
means $(A_1,\dots,A_l)\in \bigoplus_{v\in E^0} \M_{k^v_n}(K) $ is sent to
\[(\sum_{j=1}^l n_{j1}A_j,\dots,\sum_{j=1}^l n_{jl}A_j) \in
\bigoplus_{v\in E^0} \M_{k^v_{n+1}}(K),\] where $n_{ji}$ is the number of
edges connecting $v_j$ to $v_i$ and
\[\sum_{j=1}^lk_jA_j=\left(
\begin{matrix}
A_1 &       &             & \\
       & A_1 &            & \\
       &         & \ddots &\\
       &         &            & A_l\\
\end{matrix}
\right)
\]
in which each matrix is repeated $k_j$ times down the leading
diagonal and if $k_j=0$, then $A_j$ is omitted.

Writing $\LL(E)_0=\varinjlim_{n} L_{0,n}$, since the Grothendieck group $K_0$ respects the
direct limit, we have $K_0(\LL(E)_0)\cong
\varinjlim_{n}K_0(L_{0,n})$. Since $K_0$ of  (Artinian) simple
algebras are $\mathbb Z$, the ring homomorphism $L_{0,n}\rightarrow
L_{0,n+1}$ induces the group homomorphism \[\mathbb Z^{E^0} \stackrel{A_E^t}{\longrightarrow}
\mathbb Z^{E^0},\]
where $A_E^t:\mathbb Z^{E^0} \rightarrow \mathbb Z^{E^0}$ is multiplication  from left which is induced by the homomorphism~(\ref{volleyb}). 

For a finite graph $E$ with no sinks, with $n$ vertices and the adjacency matrix $A$, by Theorem~\ref{sthfin}, $K^{\gr}_0(\LL(E))\cong K_0(\LL(E)_0)$. Thus  $K^{\gr}_0(\LL(E))$ is the direct limit  of the ordered direct system 
\begin{equation}\label{thu3}
\mathbb Z^n \stackrel{A^t}{\longrightarrow} \mathbb Z^n \stackrel{A^t}{\longrightarrow}  \mathbb Z^n \stackrel{A^t}{\longrightarrow} \cdots,
\end{equation}
where the ordering in $\mathbb Z^n$ is defined point-wise.

\subsection{} \label{peeme}

Since for a finite graph with no sinks, the graded Grothendieck group of its associated Leavitt path algebra is the direct limit of the form~(\ref{thu3}), here we recall two different presentations for the direct limit of abelian groups. This will be used in Example~\ref{gfrt} and Lemma~\ref{mmpags}. 

In~\S\ref{volkhtes} it was shown that for a finite graph with no sinks  its graded Grothendieck group is a direct limit of a direct system of ordered free abelian groups with a matrix $A$ (the transpose of the adjacency matrix of the graph) acting as an order preserving group homomorphism (from the left) as follows
\begin{equation}\label{thuluncht}
\mathbb Z^n \stackrel{A}{\longrightarrow} \mathbb Z^n \stackrel{A}{\longrightarrow}  \mathbb Z^n \stackrel{A}{\longrightarrow} \cdots,
\end{equation}
where the ordering in $\mathbb Z^n$ is defined point-wise. The direct limit of this system, $\varinjlim_{A} \mathbb Z^n$, is an ordered group and  can be described as follows. Consider the pair $(a,k)$, where $a\in \mathbb Z^n$ and $k\in \mathbb N$, and define the equivalence relation $(a,k)\sim (b,k')$ if $A^{k''-k}a=A^{k''-k'}b$ for some $k'' \in \mathbb N$.  Let $[a,k]$ denote the equivalence class of $(a,k)$. Clearly $[A^na,n+k]=[a,k]$. Then it is not difficult to show that the direct limit of the system \ref{thuluncht} is the abelian group consists of equivalent classes $[a,k]$, $a\in \mathbb Z^n$, $k \in \mathbb N$, with addition defined by
\[[a,k]+[b,k']=[A^{k'}a+A^kb,k+k'].\]
The positive cone of this ordered group is the set of elements $[a,k]$, where $a\in {\mathbb Z^+}^n$, $k\in \mathbb N$. 
Furthermore, there is automorphism $\delta_A: \varinjlim_{A} \mathbb Z^n \rightarrow \varinjlim_{A} \mathbb Z^n$ defined by $\delta_A([a,k])=[Aa,k]$. 

There is another presentation for  $\varinjlim_{A} \mathbb Z^n$ which is sometimes easier to work with. Consider the set 
\begin{equation}\label{q11}
\Delta_A=\big \{v\in  A^n \mathbb  Q^n \mid A^k v \in \mathbb Z^n, \text { for some } k \in \mathbb N\big \}.
\end{equation}
 The set $\Delta_A$ forms an ordered abelian group with the usual addition of vectors and  the positive cone  
\begin{equation}\label{q22}
\Delta_A^+=\big \{v\in  A^n \mathbb Q^n \mid A^k v\in {\mathbb Z^+}^n, \text { for some } k \in \mathbb N\big \}.
\end{equation}
Furthermore, there is automorphism $\delta_A:\Delta_A\rightarrow \Delta_A$ defined by $\delta_A(v)=A v$. The map 
\begin{align}\label{kkjjhhga}
\phi:\Delta_A&\rightarrow \varinjlim_{A} \mathbb Z^n\\
v &\mapsto [A^kv,k], \notag
\end{align} 
where $k\in \mathbb N$ such that $A^k v \in \mathbb Z^n$, is an isomorphism which respects the action of $A$ and the ordering, i.e., $\phi(\Delta_A^+)= (\varinjlim_{A} \mathbb Z^n)^+$ and $\phi(\delta_A(v))=\delta_A\phi(v)$.

\begin{example} \label{gfrt}
For the graph 
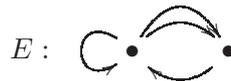
\begin{equation*}
E:\,\,\,\,\,\,\,\,\,\,\,\,{\def\labelstyle{\displaystyle}
\xymatrix{
  \bullet \ar@(lu,ld) \ar@/^0.9pc/[r] \ar@/^1.4pc/[r]  & \bullet \ar@/^0.9pc/[l]   &   
}}
\end{equation*}
with the adjacency 
$A_E=\left(
\begin{array}{cc}
 1 & 2 \\
 1 & 0
\end{array}
\right),$
the Bratteli diagram associated to $\LL(E)_0$ is 
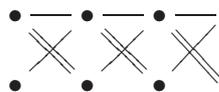
\begin{equation*}
{\def\labelstyle{\displaystyle}
\xymatrix@=15pt{
 \bullet \ar@{-}[r] \ar@{=}[dr] &  \bullet \ar@{-}[r] \ar@{=}[dr] & \bullet  \ar@{-}[r] \ar@{=}[dr] &\\
\bullet   \ar@{-}[ur] & \bullet \ar@{-}[ur]  & \bullet  \ar@{-}[ur]&
}} 
\end{equation*}
and $\LL(E)_0$ is the direct limit of the system 
\begin{align*} 
K \oplus K
 \stackrel{A_{E}^t}{\longrightarrow}  \M_2(K) &\oplus \M_2(K)
\stackrel{A_{E}^t}\longrightarrow  \M_4(K)\oplus \M_4(K)
\stackrel{A_{E}^t}\longrightarrow \cdots \\
(a,b)\mapsto \left(
\begin{array}{cc}
 a & 0 \\
 0 & b
\end{array}
\right)
&\oplus
\left(
\begin{array}{cc}
 a & 0 \\
 0 & a
\end{array}
\right)
\end{align*} 
So $K^{\gr}_0(\LL(E))$ is the direct limit of the direct system 
\begin{equation*}
\mathbb Z^2 \stackrel{A_E^t}{\longrightarrow} \mathbb Z^2 \stackrel{A_E^t}{\longrightarrow}  \mathbb Z^2 \stackrel{A_E^t}{\longrightarrow} \cdots,
\end{equation*}
Since $\det(A^t_E)=-2$, one can easily calculate that \[K^{\gr}_0(\LL(E))\cong \mathbb Z[1/2]\bigoplus  \mathbb Z[1/2].\] Furthermore 
$[\LL(E)] \in K^{\gr}_0(\LL(E))$ is represented by $(1,1)\in  \mathbb Z[1/2]\bigoplus  \mathbb Z[1/2]$. Adopting~(\ref{q11}) for the description of $K^{\gr}_0(\LL(E))$, since the action of $x$ on $K^{\gr}_0(\LL(E))$ represented by action of $A_E^t$ from the left, we have $x (a,b)=(a+b,2a)$. Furthermore, considering 
~(\ref{q22}) for the positive cone,  ${A_E^t}^k(a,b)$ is eventually positive, if $v (a,b) >0$, where $v=(2,1)$ is the Perron eigenvector of $A_E$ (see~\cite[Lemma~7.3.8]{lindmarcus}). It follows that 
\[K^{\gr}_0(\LL(E))^+=\Delta_{A_E^t}^+=\big \{(a,b) \in \mathbb Z[1/2]\oplus  \mathbb Z[1/2] \mid 2a+b > 0\big \} \cup \{(0,0)\}.\]
\end{example}

\section{Krieger's dimension groups and Wagoner's dimension modules} \label{jjhh1}

The Grothendieck group $K_0$ is a pre-ordered abelian group with the set of isomorphism classes of finitely generated projective modules as its positive cone.  For the category of ultramatricial algebras, $K_0$ along with its positive cone and the position of the identity is a complete invariant (see ~\cite{elliot} and~\cite[\S15]{goodearlbook}). Motivated by this, Krieger in~\cite{krieger1} defined an invariant for the irreducible shift of finite types. In general, a nonnegative integral $n\times n$ matrix $A$ gives rise to a stationary system. This in turn gives a direct system of ordered free abelian groups with $A$ acting as an order preserving group homomorphism as follows
\[\mathbb Z^n \stackrel{A}{\longrightarrow} \mathbb Z^n \stackrel{A}{\longrightarrow}  \mathbb Z^n \stackrel{A}{\longrightarrow} \cdots,
\]
where the ordering in $\mathbb Z^n$ is defined point-wise.  The direct limit of this system, $\Delta_A:= \varinjlim_{A} \mathbb Z^n$, (i.e, the $K_0$ of the stationary system,) along with its positive cone, $\Delta^+$, and the automorphism which induced by $A$ on the direct limit, 
$\delta_A:\Delta_A \rightarrow \Delta_A$, is the invariant considered by Krieger, now known as Krieger's dimension group. Following~\cite{lindmarcus}, we denote this triple by $(\Delta_A, \Delta_A^+, \delta_A)$.  It can be showen that two matrices $A$ and $B$ are shift equivalent if and only if their associated Krieger's dimension groups are isomorphic (\cite[Theorem~4.2]{krieger1}, and~\cite[Theorem~7.5.8]{lindmarcus}, see also~\cite[\S7.5]{lindmarcus} for a detailed algebraic treatment). Wagoner noted that the induced structure on $\Delta_A$ by the automorphism $\delta_A$ makes $\Delta_A$ a $\mathbb Z[x,x^{-1}]$-module which was systematically used in~\cite{wago1,wago2} (see also~\cite[\S3]{boyle}).

The graded Grothendieck group of a $\mathbb Z$-graded ring has a natural $\mathbb Z[x,x^{-1}]$-module structure (see~\S\ref{gftegtds}) and 
the following observation (Lemma~\ref{mmpags}) that the graded Grothendieck group of the Leavitt path algebra associated to a matrix $A$ coincides with the Krieger dimension group of the shift of finite type associated to $A^t$, i.e., the graded dimension group of a Leavitt path algebra coincides with Krieger's dimension group,
\[\big(K_0^{\gr}(\LL(E)),(K_0^{\gr}(\LL(E))^+\big ) \cong (\Delta_{A^t}, \Delta_{A^t}^+)\] 
will provide a link between the theory of Leavitt path algebras and symbolic dynamics. This lemma was also proved differently by Ara and Pardo in~\cite{arapardo}.

\begin{lemma}\label{mmpags}
Let $E$ be a finite graph with no sinks with the  adjacency matrix $A$. Then there is an isomorphism 
$\phi:K^{\gr}_0(\LL(E)) \longrightarrow \Delta_{A^t}$ such that $\phi(x \alpha)=\delta_{A^t}\phi(\alpha)$, $\alpha\in \LL(E)$, $x\in \mathbb Z[x,x^{-1}]$ and $\phi(K^{\gr}_0(\LL(E))^+)=\Delta_{A^t}^+$. 
\end{lemma}
\begin{proof}
Since by Theorem~\ref{sthfin}, $\LL(E)$ strongly graded, there is an ordered isomorphism $K^{\gr}_0(\LL(E)) \rightarrow K_0(\LL(E)_0)$. 
Thus by~\S\ref{volkhtes}, (see~(\ref{thu3})) the ordered group $K^{\gr}_0(\LL(E))$ coincides with the ordered group $\Delta_{A^t}$. We only need to check that their module structures are compatible. It is enough to  show that the action of $x$ on $K^{\gr}_0$ coincides with the action of $A^t$ on $K_0(\LL(E)_0)$, i.e., $\phi(x \alpha)=\delta_{A^t} \phi(\alpha)$. 

Set $\mathcal A=\LL(E)$. Since graded finitely generated projective modules are generated by $u\mathcal A(i)$, where $u\in E^0$ and $i \in \mathbb Z$, it suffices to show that $\phi(x [u\mathcal  A])=\delta_{A^t}\phi([u \mathcal A])$. 
Since the image of $u\mathcal A$ in $K_0(\mathcal A_0)$ is $[u\mathcal A_0]$, and $\mathcal A_0= \bigcup_{n=0}^{\infty}L_{0,n},$ (see~(\ref{ppooii})) using the presentation of $K_0$ given in~\S\ref{peeme}, we have 
\[\phi([u\mathcal  A])=[u\mathcal  A_0]=[uL_{0,0},1]=[u,1].\] Thus 
\[\delta_{A^t}\phi([u\mathcal  A])=\delta_{A^t}([u,1])=[A^t u,1]=\sum_{\{\alpha \in E^1 \mid s(\alpha)=u\}}[r(\alpha),1].\]

On the other hand, 
\begin{multline}
\phi(x[u\mathcal  A])=\phi([u\mathcal  A(1)])=\phi\big(\sum_{\{\alpha \in E^1 \mid s(\alpha)=u\}}[r(\alpha)\mathcal  A]\big)=\\ 
\sum_{\{\alpha \in E^1 \mid s(\alpha)=u\}}[r(\alpha)\mathcal  A_0]=
\sum_{\{\alpha \in E^1 \mid s(\alpha)=u\}}[r(\alpha)L_{0,0},1]=\sum_{\{\alpha \in E^1 \mid s(\alpha)=u\}}[r(\alpha),1].
\end{multline}
Thus $\phi(x [u\mathcal  A])=\delta_{A^t}\phi([u \mathcal A])$. This finishes the proof. 
 \end{proof}

It is easy to see that two matrices $A$ and $B$ are shift equivalent if and only if $A^t$ and $B^t$ are shift equivalent. Combining this with Lemma~\ref{mmpags} and the fact that Krieger's dimension group is a complete invariant for shift equivalent we have the following corollary.

\begin{corollary}\label{h99}
Let $E$ and $F$ be finite graphs with no sinks and $A_E$ and $A_F$ be their adjacency matrices, respectively. Then
$A_E$ is shift equivalent to $A_F$ if and only if there is an order preserving $\mathbb Z[x,x^{-1}]$-module  isomorphism
$K_0^{\gr}(\LL(E)) \cong_{\gr} K_0^{\gr}(\LL(F))$.
\end{corollary}

\section{Product of graphs and graded Grothendieck groups}\label{oo12}

In~\cite[p.~149]{wago2}, Wagoner considered the product of two shift spaces and showed that the dimension module of the product is isomorphic to the tensor product of the dimension module of the shift spaces. In this section we carry over this to the case of Leavitt path algebras and graded Grothendieck groups. 

Let $A$ and $B$ be the adjacency matrices of the graphs $E$ and $F$, respectively, where $|E^0|=m$ and $|F^0|=n$.
Note that $A$ and $B$ can be considered as endomorphisms in $\End_{\mathbb Z}(\mathbb Z^m)$ and $\End_{\mathbb Z}(\mathbb Z^n)$, respectively.  We define the {\it product of the graphs} $E$ and $F$, denoted by $E\otimes F$, to be the graph associated to the matrix $A\otimes B \in \End_{\mathbb Z}(\mathbb Z^{m+n})$.  Concretely, if $A=(a_{ij})_{1\leq i,j\leq m}$ and $B=(b_{lk})_{1\leq l,k\leq n}$, then 
\begin{equation}\label{poiuyt}
A\otimes B=\Big ( a_{ij}(b_{lk})_{1\leq l,k\leq n}\Big)_{1\leq i,j\leq m},
\end{equation}
 i.e., the $ij$th entry of $A\otimes B$ is the matrix block 
$(a_{ij}b_{lk})_{1\leq l,k\leq n}$. Note that this representation is independent of considering $A$ and $B$ as matrices acting from left or right.

If the matrices $A$ and $B$ have no zero rows, i.e., their associated Leavitt path algebras are strongly graded (Theorem~\ref{sthfin}), then so is the Leavitt path algebra associated to $A\otimes B$. 

\begin{example}\label{travelin}
Consider the graphs 
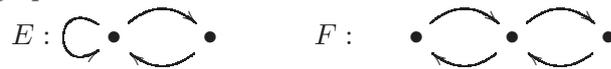
\begin{equation*}
{\def\labelstyle{\displaystyle}
E : \quad \,\, \xymatrix{
 \bullet  \ar@(lu,ld)\ar@/^0.9pc/[r] & \bullet \ar@/^0.9pc/[l] 
}} \qquad  \quad
{\def\labelstyle{\displaystyle}
F: \quad \,\, \xymatrix{
  \bullet \ar@/^0.9pc/[r] & \bullet \ar@/^0.9pc/[l]   \ar@/^0.9pc/[r] & \bullet \ar@/^0.9pc/[l] 
}} 
\end{equation*}
with their adjacency matrices
\begin{equation*}
A=\left(
\begin{array}{cc}
 1 & 1 \\
 1 & 0
\end{array}
\right)
\qquad \text{ and }  \qquad 
B=\left(
\begin{array}{ccc}
 0 & 1 & 0 \\
 1 & 0 & 1 \\
 0 & 1 & 0
\end{array}
\right). 
\end{equation*}
Then
\begin{equation*}
A\otimes B=\left(
\begin{array}{cccccc}
 0 & 1 & 0 & 0 & 1 & 0 \\
 1 & 0 & 1 & 1 & 0 & 1 \\
 0 & 1 & 0 & 0 & 1 & 0 \\
 0 & 1 & 0 & 0 & 0 & 0 \\
 1 & 0 & 1 & 0 & 0 & 0 \\
 0 & 1 & 0 & 0 & 0 & 0
\end{array}
\right)
\end{equation*}\label{pgtisu}
and the graph associated to this matrix is  
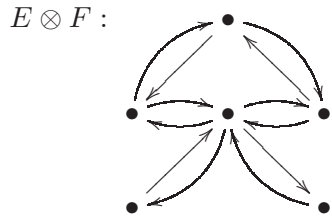
\begin{equation*}
{\def\labelstyle{\displaystyle}
E\otimes F: \xymatrix{
& \bullet \ar[dl]  \ar@/^0.8pc/[dr]\\
  \bullet \ar@/^0.8pc/[ur] \ar@/^0.4pc/[r] & \bullet \ar@/^0.4pc/[l]  \ar[dr] \ar@/^0.8pc/[dl] \ar@/^0.4pc/[r] & \bullet \ar@/^0.4pc/[l]  \ar[ul]\\
  \bullet \ar[ur] &&  \bullet \ar@/^0.8pc/[lu]
}} 
\end{equation*}

\end{example}

In the following theorem we consider the tensor product of two pre-ordered $\mathbb Z[x,x^{-1}]$-modules $G_1$ and $G_2$, over $\mathbb Z$. We define the action of $x$ on $G_1\otimes_{\mathbb Z} G_2$ diagonally, i.e., $x (g_1\otimes g_2)=xg_1 \otimes xg_2$ and extend it to the whole $\mathbb Z[x,x^{-1}]$ naturally. This makes $G_1\otimes_{\mathbb Z} G_2$ a  $\mathbb Z[x,x^{-1}]$-module. Further, the monoid $G_1^{+} \otimes G_2^{+}$ (i.e., the set of direct sums of images of $G_1^{+}$ and $G_2^{+}$ in $G_1\otimes_{\mathbb Z} G_2$ ), makes $G_1\otimes_{\mathbb Z} G_2$ an pre-order group respecting the module structure.

\begin{theorem}\label{ppsch}
Let $\LL(E)$ and $\LL(F)$ be Leavitt path algebras associated to finite graphs with no sinks $E$ and $F$. 
Then there is an order preserving $\mathbb Z[x,x^{-1}]$-module isomorphism 
\[K^{\gr}_0(\LL(E\otimes F)) \cong K^{\gr}_0(\LL(E))\otimes K^{\gr}_0(\LL(F)),\]
which sends  $[\LL(E\otimes F)]$ to $ [\LL(E)]\otimes [\LL(F)]$. 
\end{theorem}
\begin{proof}
Let $A_E$ and $A_F$ be the adjacency matrices of the graphs $E$ and $F$, respectively, where $|E^0|=m$ and $|F^0|=n$.
Set $A=A_E^t$ and $B=A_F^t$. Observe that $A\otimes B= (A_E \otimes A_F)^t$. Now by~(\ref{thu3}) the vertical maps of the following diagram are isomorphisms. 
\begin{equation*}
\xymatrix@=15pt{
K^{\gr}_0(\LL(E)) \otimes K^{\gr}_0(\LL(F)) \ar[d] \ar@{.>}[rr] && K^{\gr}_0(\LL(E\otimes F)) \ar[d] \\
 \varinjlim_{A} \mathbb Z^m \otimes  \varinjlim_{B} \mathbb Z^n \ar[rr]^{\phi} &&   \varinjlim_{A\otimes B} \mathbb Z^{m+n}}
\end{equation*}
Using the description of~\S\ref{peeme} for the direct limits, define $\phi$ on generators as follows: \[\phi([a,k]\otimes[b,l])=[A^{l}a\otimes B^kb,k+l].\]  One checks easily that this map is well-defined and is a homomorphism of groups. Further,
 \begin{multline*}
 x\phi([a,k]\otimes[b,l])=x[A^{l}a\otimes B^kb,k+l]= \\
 \big[(A\otimes B)(A^{l}a\otimes B^kb) ,k+l\big]=
[A^{l+1}a\otimes B^{k+1}b,k+l]=\\\phi( [Aa,k]\otimes[Bb,l])=\phi\big(x([a,k]\otimes[b,l])\big),\end{multline*}
shows that $\phi$ is a $\mathbb Z[x,x^{-1}]$-module homomorphism. Define 
\begin{align*}
\psi: \varinjlim_{A\otimes B} \mathbb Z^{m+n} &\longrightarrow \varinjlim_{A} \mathbb Z^m \otimes  \varinjlim_{B} \mathbb Z^n\\
[c,k]&\longmapsto \sum_i [a_i,k]\otimes[b_i,k],
\end{align*}
where $f(c)=\sum_i a_i\otimes b_i$ under a natural isomorphism $f:\mathbb Z^{m+n}\rightarrow \mathbb Z^m\otimes \mathbb Z^n$. One can check that $\psi$ is indeed well-defined and $\phi\psi$ and $\psi\phi$ are the identity map of the corresponding groups. 

Finally, $[\LL(E]\otimes [\LL(F)]$  is represented by $[\bar 1,0]\otimes [\bar 1,0]$ in $\varinjlim_{A} \mathbb Z^m \otimes  \varinjlim_{B}\mathbb Z^n$ and
$\phi([\bar1,0]\otimes [\bar1,0])=[\bar1,0]$ which represents  $[\LL(E\otimes F)]$ in $K^{\gr}_0(\LL(E\otimes F))$. 
\end{proof}

\begin{example}
For the graph $E$ in the Example~\ref{travelin}, one can calculate its dimension group as follows:
\begin{align*}
K^{\gr}_0(\LL(E)) & \cong\Z\oplus \Z; \\
K^{\gr}_0(\LL(E))^{+} & \cong\big \{(a,b) \mid \frac{1+\sqrt{5}}{2} a+ b\geq0 \big \}; \\
[\LL(E)]&=(1,1);\\
x(a,b)&=(a+b,a).
\end{align*} 
Furthermore for the graph $F$, we have 
\begin{align*}
K^{\gr}_0(\LL(F)) & \cong\Z[1/2]\oplus \Z[1/2]; \\
K^{\gr}_0(\LL(F))^{+} & \cong\N[1/2]\oplus \N[1/2]; \\
[\LL(F)]&=(2,1);\\
x(a,b)&=(2b,a).
\end{align*} 
This information along with Theorem~\ref{ppsch}, will easily determine the graded dimension group associated to the graph $E\otimes F$ in Example~\ref{pgtisu}. 

\end{example}

\section{Conjugacy, Shift equivalence and Graded Morita equivalence}\label{willi}

\subsection{}\label{mooiw}
The notion of the shift equivalence for matrices was introduced by Williams~\cite{williams} (see also~\cite[\S7]{lindmarcus}) in an attempt to provide a computable machinery for determining the conjugacy between two shift of finite types. Recall that two square nonnegative integer matrices $A$ and $B$ are called {\it elementary shift equivalent}, and denoted by $A\sim_{ES} B$, if there are nonnegative matrices $R$ and $S$ such that $A=RS$ and $B=SR$. 
The equivalence relation $\sim_S$  on square nonnegative integer matrices generated by elementary strong shift equivalence is called {\it strong shift equivalence}. The weaker notion of shift equivalent is defined as follows. The nonnegative integer matrices $A$ and $B$ are called {\it shift equivalent} if there are nonnegative matrices $R$ and $S$ such that $A^l=RS$ and $B^l=SR$, for some $l\in \mathbb N$, and 
$AR=RB$ and $SA=BS$.

\subsection{} \label{hdsweg}
We recall the type of the graphs which are of interest in symbolic dynamics. It turns out the Leavitt path algebras associated to this class of graphs are very interesting algebras (i.e, purely infinite simple algebras). 

Let $E$ be a finite directed graph. Then $E$ is {\it irreducible} if given any two vertices $v$ and $w$ in $E$, there is a path from 
$v$ to $w$. $E$ is called {\it essential} if there are neither sources nor sinks in $E$, and
$E$ is {\it trivial} if $E$ consists of a single cycle with no other vertices or edges (see ~\cite [Definition~2.2.13]{lindmarcus} and~\cite{franks}). 
A set of graphs which is  simultaneously irreducible, essential, and nontrivial is of great interest in the theory of shifts of finite type. 
Indeed, an edge that begins at a source or ends at a sink does not appear in any bi-infinite path so the only part of an arbitrary finite graph $E$ which appears in symbolic dynamic is the graph with no sources and sinks which obtained by repeatedly removing all the sources and  sinks from $E$ (see~\cite[The remark after Example~2.2.8 and Proposition~2.2.10]{lindmarcus}). The following result connects this class of graphs to a very interesting class of Leavitt path algebras. 
Let $E$ be a finite graph. Then $E$ is irreducible, nontrivial, and essential if and only if 
$E$ contains no sources, and $\LL(E)$ is purely infinite simple (see~\cite[Lemma~1.17]{flowa}).

Starting from a graph $E$, and a partition $\mathcal P$ of the edges, one can obtain new graphs, called {\it out-splitting}, denoted by $E_s(\mathcal P)$, and {\it in-splitting}, denoted by $E_r(\mathcal P)$. The converse of this processes are called {\it out-amalgamation} and {\it in-amalgamation} (see~\cite[\S2.4]{lindmarcus}, and \cite[Definition~1.9 and Definition~1.12]{flowa}). Furthermore, when the graph has a source, say $v$, there is a source elimination graph $E_{\backslash v}$, which is obtained by removing $v$ and all the edges emitting from $v$ from $E$ (see~\cite[Definition~1.2]{flowa}).

The following observation will be used in Proposition~\ref{hgysweet}. Let $E$ be a finite graph, let $v \in E^0$, and let $\mathcal P$ be a partition of the edges of $E$. Then $E$ is essential (resp. nontrivial, resp. irreducible) if and only if $E_s(\mathcal P)$, $E_r(\mathcal P)$, and $E_{\backslash v}$ are each essential (resp. nontrivial, resp. irreducible) (see~\cite[Lemma~1.16]{flowa}).

In~\cite[Proposition~1.4]{flowa} it was shown that for a finite graph $E$ such that $\LL(E)$ is simple, removing a source vertex would not change the category of the corresponding Leavitt path algebra up to the Morita equivalent. We need a similar result in the graded setting without an extra assumption of simplicity. Thanks to Theorem~\ref{grmorim11}, once this is proved it gives the nongraded statement naturally. 

\begin{proposition}\label{valenjov}
Let $E$ be a finite graph with no sinks and at least two vertices. Let $v\in E^0$ be a source. Then $\LL(E_{\backslash v})$ is graded Morita equivalent to $\LL(E)$. Consequently $\LL(E_{\backslash v})$ is Morita equivalent to $\LL(E)$. 
\end{proposition}
\begin{proof}
Since $E_{\backslash v}$ is a complete subgraph of $E$, there is a  graded algebra homomorphism $\phi:\LL(E_{\backslash v}) \rightarrow \LL(E)$, such that $\phi(u)=u$, $\phi(e)=e$ and $\phi(e^*)=e^*$, where $u \in E^0 \backslash \{v\}$ and $e\in E^1$.  
The graded uniqueness theorem~\cite[Theorem~4.8]{tomforde} implies  $\phi$ is injective. Thus  $\LL(E_{\backslash v})  \cong_{\gr} 
\phi(\LL(E_{\backslash v}))$. It is not difficult to see that $\phi(\LL(E_{\backslash v})) =p\LL(E) p$, where $p=\sum_{u \in E_{\backslash v}^0}u$. 
This immediately implies that $\LL(E_{\backslash v})$ is  graded Morita equivalent to $p\LL(E)p$. On the other hand, the (graded) ideal generated by $p=\sum_{u \in E_{\backslash v}^0}u$ coincides with the (graded) ideal generated by $\{u \mid u \in E_{\backslash v}^0 \}$. But the smallest hereditary and saturated subset of $E^0$ containing $E_{\backslash v}^0$ is $E^0$. Thus by~\cite[Theorem~5.3]{amp} the ideal generated by $\{u \mid u \in E_{\backslash v}^0 \}$ is $\LL(E)$. This shows $p$ is a full homogeneous idempotent in $\LL(E)$. 
Thus $p\LL(E)p$ is graded Morita equivalence to $\LL(E)$ (see Example~\ref{idempogr}). Putting these together we get $\LL(E_{\backslash v})$ is graded Morita equivalent to $\LL(E)$. The last part of the Proposition follows from the Green-Gordon theorem (see~\S\ref{hgdeiii}). 
\end{proof}

We are in a position to relate the notion of (strongly) shift equivalent in matrices with the graded Morita theory of Leavitt path algebras associated to these matrices. 

\begin{proposition}\label{hgysweet}\hfill
\begin{enumerate}[\upshape(1)]
\item Let $E$ be an essential graph and $F$ be a graph obtained from an in-splitting or out-splitting of the graph $E$. Then 
$\LL(E)$ is graded Morita equivalent to $\LL(F)$. 

\item For essential graphs $E$ and $F$, if the adjacency matrices $A_E$ and $A_F$  are strongly shift equivalent then $\LL(E)$ is graded Morita equivalent to $\LL(F)$.

\item  For graphs $E$ and $F$ with no sinks, if $\LL(E)$ is graded Morita equivalent to $\LL(F)$, then the adjacency matrices $A_E$ and $A_F$  are shift equivalent.

\end{enumerate}
\end{proposition}
\begin{proof}
(1) First suppose $E$ is an essential graph and $E_r(\mathcal P)$ the in-split graph from $E$ using a partition $\mathcal P$. 
For each $v\in E^0$, define $Q_v=v_1$, which exists by the assumption that $E$ has no sources. For $e \in \mathcal E_i^v$, define $T_e=\sum_{f\in s^{-1}(v)} e_1f_if^*_1$ and $T^*_e=\sum_{f\in s^{-1}(v)} f_1f^*_ie_1^*$. 
In~\cite[Proposition~1.11]{flowa}, it was proved that $\{Q_v,T_e,T^*_e \mid v\in E^0,e\in E^1\}$ is an $E$-family which in turn induces a $K$-algebra homomorphism 
\begin{align}\label{jonhte1}
\pi:\LL(E) &\longrightarrow \LL(E_r(\mathcal P)),\\
v &\longmapsto Q_v=v_1, \notag\\
e &\longmapsto T_e=\sum_{f\in s^{-1}(v)} e_1f_if^*_1, \notag \\
e^* &\longmapsto T^*_e=\sum_{f\in s^{-1}(v)} f_1f^*_ie_1^*. \notag
\end{align}
Furthermore, it was shown that $\pi(\LL(E))=p \LL(E_r(\mathcal P))p$ where $p=\pi(1_{\LL(E)})=\sum_{v\in E^0}v_1$.

Since $\pi(v)=v_1 \not = 0$ (see~\cite[Lemma~1.5]{goodearl}), the graded uniqueness theorem~\cite[Theorem~4.8]{tomforde} implies  
$\pi$ is injective. Furthermore~(\ref{jonhte1}) shows that $\pi$ is a graded map. Thus 
\[\LL(E) \cong_{\gr} p \LL(E_r(\mathcal P))p.\] 
We will show that $p$ is a full idempotent in $\LL(E_r(\mathcal P))$. The (graded) ideal generated by $p=\sum_{v\in E^0}v_1$ coincides with the (graded) ideal generated by $\{v_1 \mid v \in E^0 \}$. But the smallest hereditary and saturated subset of $E_r(\mathcal P)^0$ containing $\{v_1 \mid v \in E^0 \}$ is $E_r(\mathcal P)^0$. Thus by~\cite[Theorem~5.3]{amp} the ideal generated by $\{v_1 \mid v \in E^0 \}$ is $\LL(E_r(\mathcal P))$. This shows $p$ is a full homogeneous idempotent in $\LL(E)$. 
Now in Theorem~\ref{grmorim11}(4) by setting $e=p$, $n=1$, $B=\LL(E)$ and $A=\LL(E_r(\mathcal P))$, we get  that $\LL(E)$ is graded Morita equivalent to $\LL(E_r(\mathcal P))$. 

On the other hand, if $E_s(\mathcal P)$ is the out-split graph from $E$ using a partition $\mathcal P$, then by~\cite[Theorem~2.8]{aalp}, there is a graded $K$-algebra isomorphism $\pi:\LL(E)\rightarrow  \LL(E_s(\mathcal P)).$ Again, Theorem~\ref{grmorim11}(4) implies that $\LL(E)$ is graded Morita equivalent to $\LL(E_s(\mathcal P))$.

(2) If $A_E$ is strongly shift equivalent to $A_F$, a combination of the Williams theorem~\cite[Theorem~7.2.7]{lindmarcus} and the Decomposition theorem~\cite[Theorem~7.1.2, Corollary~7.1.5]{lindmarcus} implies that the graph $F$ can be obtained from $E$  by a sequence of out-splittings, in-splittings, out-amalgamations, and in-amalgamation. All the graphs appear in this sequence are essential (see~\S\ref{hdsweg}). Now a repeated application of part (1) gives that $\LL(E)$ is graded Morita equivalent to $\LL(F)$. 


(3) Since $\Gr \LL(E) \approx_{\gr} \Gr \LL(F)$, there is an order preserving $\mathbb Z[x,x^{-1}]$-module  isomorphism
$K_0^{\gr}(\LL(E)) \cong_{\gr} K_0^{\gr}(\LL(F))$ (see~\S\ref{hgdeiii}). Thus by Corollary~\ref{h99}, $A_E$ and $A_F$  are shift equivalent.
\end{proof}

\begin{remark}
Proposition~\ref{hgysweet}(3) shows that if $\LL(E) \approx_{\gr} \LL(F)$ then  the adjacency matrices of $E$ and $F$ are shift equivalent. One thinks that the converse of this statement is also valid. In fact, if $A_E$ is shift equivalent to $A_F$ then by Corollary~\ref{h99}, there is an order preserving $\mathbb Z[x,x^{-1}]$-module isomorphism $K_0^{\gr}(\LL(E))  \cong K_0^{\gr}(\LL(F))$. Since $\LL(E)$ and $\LL(F)$ are strongly graded, this implies  $K_0(\LL(E)_0)  \cong K_0(\LL(F)_0)$ as partialy ordered abelian groups. Since $\LL(E)_0$ and $\LL(F)_0$ are ultramatricial algebras, by~\cite[Coroallary~15.27]{goodearlbook} $\LL(E)_0$ is Morita equivalent to $\LL(F)_0$. Now the following diagram shows that $\Gr \LL(E)$ is equivalent to $\Gr \LL(F)$.
\begin{equation*}
\xymatrix{
\Modd \LL(E)_0 \ar[rr]   \ar[d]_{-\otimes \LL(E)} && \Modd \LL(F)_0 \ar[d]^{-\otimes \LL(F)}\\
\Gr \LL(E) \ar[rr]&& \Gr \LL(F) .
}
\end{equation*}
However it is not clear whether this equivalence is graded as in Definition~\ref{grdeffsa}(1). 

\end{remark}

Recall that for a graph, the associated Leavitt path algebra is purely infinite simple unital, if and only if the graph is finite, any vertex is connected to a cycle and any cycle has an exist (see~\cite{aap06} and~\cite[p.~205]{flowa}). Note that for a finite graph, the condition of not having a sink is equivalent to any vertex be connected to a cycle. Thus by Theorem~\ref{sthfin}, purely infinite simple unital Leavitt path algebras are strongly graded.

\begin{example}[{\sc Purely infinite simple and its transpose are not graded Morita Equivalent}]
We have seen most of the results already proved in the literature on Morita equivalence, such as in-splitting, out-splitting, and removing of the sources can be extended to a stronger graded Morita equivalence. However this is not always the case. In~\cite[Proposition~3.10]{flowa}, it was shown that for a finite graph $E$ without sources such that $\LL(E)$ is purely infinite simple, $\LL(E)$ and $\LL(E^{\op})$ are Morita equivalent. (Here $E^{\op}$ is the {\it opposite} or {\it transpose} of the graph $E$, i.e., $E^{\op}$ is obtained from $E$ by reversing the arrows, so  $A_E^t=A_{E^{\op}}$. In~\cite{flowa}, $E^{\op}$ is denoted by $E^t$.) However there are examples of the graph $E$ such that $\LL(E)$ and $\LL(E^{\op})$ are not graded Morita equivalent. Consider the graph $E$ with the adjacency matrix  
\begin{equation*}
A_E=\left(
\begin{array}{cc}
 19 & 5 \\
 4 & 1
\end{array}
\right).
\end{equation*}
The Leavitt path algebra $\LL(E)$ is purely infinite simple unital algebra with no sources. If $\LL(E)$ is graded Morita equivalent to $\LL(E^{\op})$, by Proposition~\ref{hgysweet}(3),
$A_E$ and $A_{E^{\op}}=A^t_E$ are shift equivalent. But it is known that $A_E$ and $A^t_E$ are not shift equivalent (see~\cite[Example~7.4.19]{lindmarcus}).
\end{example}

\begin{example}[\sc{Strongly shift equivalent does not imply isomorphism}]\label{hhyyuu}

By Proposition~\ref{hgysweet} if two essential graphs are strongly shift equivalent then their associated Leavitt path algebras 
are graded Morita equivalent. The following example shows that strongly shift equivalent, however, does not imply the Leavitt path algebras are isomorphic. 

Consider the graphs 

\begin{equation}
E:\,\,\,\,\,\,\,\,\,\,\,\,{\def\labelstyle{\displaystyle}
\xymatrix{
  \bullet \ar@(lu,ld) \ar@/^0.9pc/[r] \ar@/^1.4pc/[r]  & \bullet \ar@/^0.9pc/[l]   &   
}}
E^{\op}:\,\,\,{\def\labelstyle{\displaystyle}
\xymatrix{
  \bullet  \ar@/^0.9pc/[r] \ar@/^1.4pc/[r]  & \bullet \ar@/^0.9pc/[l]  \ar@(ru,rd) &   
}} 
\end{equation}

The adjacency matrices of $E$ and $E^{\op}$ are strongly shift equivalent as the following computation shows. Thus by Proposition~\ref{hgysweet}, $\LL(E)\approx_{\gr}\LL(E^{\op})$. However we will show $\LL(E) \not \cong_{\gr} \LL(E^{\op})$.  

First,  
\(A_E={\left(
\begin{array}{cc}
 1 & 2 \\
 1 & 0
\end{array}
\right)}\) 
and 
\(A_{E^{\op}}={\left(
\begin{array}{cc}
 1 & 1 \\
 2 & 0
\end{array}
\right)}\). Notice that $A_E=A_{E^{\op}}^t$. 
Let  \(R_1={\left(
\begin{array}{ccc}
 1 & 1 & 0 \\
 0 & 0 & 1
\end{array}
\right)}\) and 
\(S_1={\left(
\begin{array}{cc}
 1 & 1 \\
 0 & 1 \\
 1 & 0
\end{array}
\right)}\). Then $A_E=R_1 S_1$ and set $E_1:=S_1 R_1={\left(
\begin{array}{ccc}
 1 & 1 & 1 \\
 0 & 0 & 1 \\
 1 & 1 & 0
\end{array}
\right)}$.
Let 
\(R_2={\left(
\begin{array}{ccc}
 0 & 1 & 1 \\
 1 & 0 & 0 \\
 0 & 0 & 1
\end{array}
\right)}\) and 
\(S_2={\left(
\begin{array}{ccc}
 0 & 0 & 1 \\
 0 & 0 & 1 \\
 1 & 1 & 0
\end{array}
\right)}\). Then $E_1=R_2 S_2$ and set $E_2:=S_2 R_2=
{\left(
\begin{array}{ccc}
 0 & 0 & 1 \\
 0 & 0 & 1 \\
 1 & 1 & 1
\end{array}
\right)}$. Finally, let 
\(R_3={\left(
\begin{array}{cc}
 1 & 0 \\
 1 & 0 \\
 1 & 1
\end{array}
\right)}\) and 
\(S_3={\left(
\begin{array}{ccc}
 0 & 0 & 1 \\
 1 & 1 & 0
\end{array}
\right)}\).  
Then $E_2=R_3S_3$ and $A_{E^{\op}}=S_3R_3$. This shows \[A_E\sim_{ES} E_1 \sim_{ES} E_2 \sim_{ES} A_{E^{\op}}.\]
Thus $A_E\sim_S A_{E^{\op}}$. 

Now suppose there is an order $\mathbb Z[x,x^{-1}]$-module isomorphism $\phi:K^{\gr}_0(\LL(E)) \rightarrow  K^{\gr}_0(\LL(E^{\op}))$. Since as abelian groups $K^{\gr}_0(\LL(E))\cong K^{\gr}_0(\LL(E^{\op}))\cong \mathbb Z[1/2]\bigoplus  \mathbb Z[1/2]$, $\phi$ is a $2\times2$ matrix of the form 
$\left(
\begin{array}{cc}
 a & b \\
 c & d 
\end{array}
\right)$, 
where $a,b,c,d\in \mathbb Q$. Furthermore, $[\LL(E)]=(1,1)\in  \mathbb Z[1/2]\bigoplus  \mathbb Z[1/2]$ and similarly, 
$[\LL(E^{\op})]=(1,1)\in  \mathbb Z[1/2]\bigoplus  \mathbb Z[1/2]$.
But then $\phi(A_{E^t} (1,1))=A_{{E^{\op}}^t} \phi(1,1)=A_{{E^{\op}}^t} (1,1)$ (see also Example~\ref{gfrt}). A quick calculation now shows that this is not possible. 


This example shows that although there is an order isomorphism $K^{\gr}_0(\LL(E)) \cong K^{\gr}_0(\LL(E^{\op}))$ as $\mathbb Z[x,x^{-1}]$-modules, but $\LL(E) \not \cong_{\gr} \LL(E^{\op})$.  This implies that in the classification conjecture (see~\S\ref{conjisi}), the assumption of pointed isomorphisms can't be relaxed. 
\end{example}

We are in a position to settle the graded conjecture (see~\S\ref{conjisi} and~\cite[Conjecture~1]{hazann})  for the case of purely infinite simple Leavitt path algebras up to grading. Namely, we are able to show that if the graded dimension groups are isomorphic, then the Leavitt path algebras are isomorphic. The theorem guarantees an isomorphism
between the algebras, but  not a graded isomorphism. 

\begin{theorem}\label{thucomingtobal}
Let $E$ and $F$ be graphs such that $\LL(E)$ and $\LL(F)$ are purely infinite simple unital algebras. Then $\LL(E) \cong \LL(F)$ if  there is an order preserving $\mathbb Z[x,x^{-1}]$-module isomorphism
\begin{equation}\label{noperfect}
\big (K_0^{\gr}(\LL(E)),[\LL(E)]\big ) \cong \big (K_0^{\gr}(\LL(F)),[\LL(F)]\big ).
\end{equation}
\end{theorem}
\begin{proof}
Equation~\ref{noperfect} gives an isomorphism  
 $\phi: K^{\gr}_0(\LL(E)) \rightarrow K^{\gr}_0(\LL(F)),$ where $[\LL(E)]$ is sent to $[\LL(F)].$ 
The main result of \cite{haziso} shows that the diagram below is commutative, where $U$ is the forgetful functor and $T_1$ is the suspension functor and the isomorphism $\phi_1$ is induced from the commutativity of the left diagram.   
\begin{equation*}
\xymatrix@=15pt{
K^{\gr}_0(\LL(E)) \ar[d]^{\phi} \ar[r]^{T_1-\id} &
K^{\gr}_0(\LL(E)) \ar[d]^{\phi} \ar[r]^{U}& K_0(\LL(E))\ar[d]^{\phi_1} \ar[r] &0\\
K^{\gr}_0(\LL(F)) \ar[r]^{T_1-\id}  &
K^{\gr}_0(\LL(F)) \ar[r]^{U}& K_0(\LL(F)) \ar[r] &0.}
\end{equation*}
 Thus we have an induced isomorphism  
 \begin{align}\label{hvvveeo}
 \phi_1: K_0(\LL(E)) &\longrightarrow K_0(\LL(F))\\
 [\LL(E)] &\longmapsto [\LL(F)].\notag
 \end{align}

Next we show that $\LL(E)$ is Morita equivalent to $\LL(F)$. Let $E'$ and $F'$ be graphs with no sources by repeatedly removing the sources from $E$ and $F$, respectively. By repeated application of Proposition~\ref{valenjov} we have $\LL(E) \approx_{\gr} \LL(E')$ and $\LL(F) \approx_{\gr} \LL(F')$, which in turn shows that 
there are order preserving $\mathbb Z[x,x^{-1}]$-modules isomorphisms  $K^{\gr}_0(\LL(E)) \cong  K^{\gr}_0(\LL(E'))$ and $K^{\gr}_0(\LL(F)) \cong  K^{\gr}_0(\LL(F'))$. Combining these with Equation~\ref{noperfect}, we get an order preserving $\mathbb Z[x,x^{-1}]$-modules isomorphism $K^{\gr}_0(\LL(E')) \cong  K^{\gr}_0(\LL(F'))$. 
Now Corollary~\ref{h99} implies that $A_{E'}$ is shift equivalent to $A_{F'}$. By~\cite[Theorem~7.4.17]{lindmarcus}, 
$\BF(A_{E'})\cong \BF(A_{F'})$, where $\BF$ stands for the Bowen-Franks groups. On the other hand, by~\cite[Exercise~7.4.4]{lindmarcus},
$\det(1-A_{E'})=\det(1-A_{F'})$. Since the matrices $A_{E'}$ and $A_{F'}$ are irreducible, the main theorem of Franks~\cite{franks} gives that $A_{E'}$ is flow equivalent to $A_{F'}$. Thus $A_{F'}$ can be obtained from $A_{E'}$ by a finite sequence of 
in/out-splitting and expansion of graphs (see~\cite{parrysullivan}). 
Each of these transformation preserve Morita equivalence (see the proof of Theorem~1.25 in~\cite{flowa}).  So $\LL(E')$ is Morita equivalent to $\LL(F')$. Again using Proposition~\ref{valenjov}, we get that  
$\LL(E)$ is Morita equivalent to $\LL(F)$. Now by~\cite[Theorem~2.5]{flowa}, 
this Morita equivalent together with $K$-group isomorphism~(\ref{hvvveeo}),  gives  $\LL(E) \cong \LL(F)$. 
\end{proof}

\section{Noncommutative algebraic geometry} \label{noncomui}

Let $K$ be a field. If $R$ is a commutative $K$-algebra which is generated by a finite number of elements of degree 1, then by  the celebrated work of Serre~\cite{serre}, 
the category of quasi-coherent sheaves on the scheme $\PProj(R)$ is equivalent to $\QGr R:= \Gr R/\Fdim R$, where $\Gr R$ is the category of $\mathbb Z$-graded modules over $R$ and $\Fdim R$ is the Serre subcategory of (direct limit of) finite dimensional submodules.  In particular when $R=K[x_0,x_1,\dots,x_n]$, then 
$\Qcoh \mathbb P^n$ is equivalent to  $\QGr K[x_0,x_1,\dots,x_n]$ 

Inspired by this, noncommutative algebraic geometry associates to a $\mathbb Z$-graded $K$-algebra $A$ with support $\mathbb N$, a ``noncommutative scheme'' $\PProj_{nc}(A)$ that is defined implicitly by declaring that the category of Òquasi-coherent sheavesÓ on $\PProj_{nc}(A)$ is 
$\QGr A:=\Gr A /\Fdim A$. When $A$ is coherent and $\grr A$ its category of finitely presented graded modules then $\qgr A := \grr A/\fdim A$ is viewed as the category of Òcoherent sheavesÓ on $\PProj_{nc}(A)$ (see~\cite{serre,artin,smith1} for more precise statements). 

For a finite graph $E$, Paul Smith~\cite{smith2} gave a description of the category $\QGr \mathcal P(E)$, where $\mathcal P(E)$ is the path algebra  associated to $E$,  in terms of easier to study categories of graded modules over Leavitt path algebras and ultramatricial algebras. Note that free algebras (on $n$ generators) are examples of path algebras (of the graph with one vertex and $n$ loops). Further, 
in~\cite{smith3} he showed that for two finite graphs $E$ and $F$ with no sinks or sources, if  their adjacency matrices are shift equivalent,  then the ``noncommutative schemes'' represented by their path algebras are the same, i.e.,  $\QGr \mathcal P(E) \cong \QGr \mathcal P(F)$.

Recall that, by assigning $1$ to edges and $0$ to vertices, the path algebra $\mathcal P(E)$ is a $\mathbb Z$-graded algebra with support $\mathbb N$. The category of $\mathbb Z$-graded right $\mathcal P(E)$-modules with degree-preserving homomorphisms is denoted by $\Gr \mathcal P(E)$ and we write $\Fdim \mathcal P(E)$ for its full subcategory of modules that are the sum of their finite-dimensional submodules. Since $\Fdim \mathcal P(E)$ is a localising subcategory of $\Gr \mathcal P(E)$ we can form the quotient category
\[\QGr \mathcal P(E) := \Gr \mathcal P(E) / \Fdim \mathcal P(E).\]

\begin{theorem}[Paul Smith~\cite{smith2}, Theorem~1.3]\label{operahouse}
Let $E$ be a finite graph  and let $E'$ be the graph without sources and sinks that is obtained by repeatedly removing all sources and sinks  from $E$. Then
\[ \QGr \mathcal P(E) \approx \Gr \LL(E') \approx \Modd \LL(E')_0.\]
\end{theorem}

The following theorem shows that the graded Grothendieck group can be considered as a complete invariant for the quotient category of path algebras. 
\begin{theorem}\label{algnihy}
Let $E$ and $F$ be graphs with no sinks. Then  $K^{\gr}_0(\LL(E)) \cong K^{\gr}_0(\LL(F))$ as ordered abelian groups if and only if 
$\QGr \mathcal P(E)  \approx \QGr \mathcal P(F)$.
\end{theorem}
\begin{proof}
Using Dade's theorem (see~\S\ref{volkhtes5}), and the fact that $\LL(E)$ and $\LL(F)$ are strongly graded (Theorem~\ref{sthfin}), as ordered abelian groups, $K^{\gr}_0(\LL(E)) \cong K^{\gr}_0(\LL(F))$  if and only if 
$K_0(\LL(E)_0) \cong K_0(\LL(F)_0)$. On the other hand, since $\LL(E)_0$ and $\LL(F)_0$ are ultramatricial algebras (see~\S\ref{volkhtes}), by~\cite[Coroallary~15.27]{goodearlbook} $K_0(\LL(E)_0) \cong K_0(\LL(F)_0)$ as ordered abelian groups if and only if 
$\LL(E)_0$ is Morita equivalent to $\LL(F)_0$. Furthermore, a repeated application of Proposition~\ref{valenjov} shows that repeatedly removing all the sources from $E$ and $F$ (call the new graphs obtained in this way with no sources $E'$ and $F'$, respectively) do not alter the corresponding categorise modulo equivalence, i.e., $\LL(E) \approx_{\gr} \LL(E')$ and $\LL(F) \approx_{\gr} \LL(F')$, which implies 
$\LL(E)_0 \approx \LL(E')_0$ and $\LL(F)_0 \approx \LL(F')_0$. 

Now Combining this with Theorem~\ref{operahouse} we have $K^{\gr}_0(\LL(E)) \cong K^{\gr}_0(\LL(F))$ if and only if $\QGr \mathcal P(E)  \approx \QGr \mathcal P(F)$.
\end{proof}

Theorem~\ref{algnihy} gives the following corollary which is the main result of~\cite{smith3}.

\begin{corollary}\label{hghghg}
Let $E$ and $F$ be graphs with no sinks. If $A_E$ is shift equivalent to $A_F$ then  $\QGr \mathcal P(E)  \approx \QGr \mathcal P(F)$. 
\end{corollary}
\begin{proof}
Since $A_E$ is shift equivalent to $A_F$, the Krieger's dimension groups associated with $E$ and $F$ are isomorphism. But by 
Lemma~\ref{mmpags} the Krieger dimension group coincides with the graded Grothendieck group, and so the corollary follows from Theorem~\ref{algnihy}.
\end{proof}

The converse of Corollary~\ref{hghghg} is not valid as the following example shows. Let $E$ be a graph with one vertex and two loops and $F$ be a graph with one vertex and four loops. Then $K^{\gr}_0(E)=\mathbb Z[1/2]$ and $K^{\gr}_0(F)=\mathbb Z[1/2]$. So the identity map gives an order preserving group isomorphism between the $K^{\gr}_0$-groups. (Note that this isomorphism is not $\mathbb Z[x,x^{-1}]$-module isomorphism.) Then Theorem~\ref{algnihy} shows that $\QGr \mathcal P(E)  \approx \QGr \mathcal P(F)$. However one can readily see that $A_E$ is not shift equivalent to $A_F$.


\section{Appendix}

In Proposition~\ref{hgysweet} we observed that the Leavitt path algebras associated to out-splitting and in-splitting graph are graded Morita equivalent to the Leavitt path algebra of the original graph. These have been shown separately. There is a unified manner to do this, using the Ashton and Bates notion of elementary shift equivalent defined for two graphs. Namely, if two graphs $E_1$ and $E_2$ are elementary shift equivalent via a graph $E_3$, then $\LL(E_1)$ and $\LL(E_2)$ are graded Morita equivalent. It is not difficult to see that the out-splitting and in-splitting are elementary shift equivalent to the original graph~\cite{batespask}. However, the price to be paid for this unified approach is, due to the construction of $E_3$, we need to change the grading of Leavitt path algebras and the eventual graded Morita equivalent is $(1/2)\mathbb Z$-graded. 

The following definition given in~\cite{ashton,bates02} for directed graphs with no sinks, provides graph theoretical conditions, when the adjacency matrices of two graphs are elementary shift equivalent 
(see also~\cite[p.~227]{lindmarcus}).

\begin{deff}\label{ttjeaong}
Let $E_i = (E^0_i,E^1_i,r_i,s_i)$, for $i = 1,2$, be graphs. Suppose there is a graph $E_3 = (E^0_3, E^1_3, r_3, s_3)$ such that
\begin{enumerate}
\item $E^0_3 =E^0_1 \cup E^0_2$  and $E^0_1 \cap E^0_2 =\emptyset$,

\item $E^1_3 =E^1_{12} \cup E^1_{21}$, where $E^1_{ij} := \{ e \in E^1_3 \mid s_3(e) \in E^0_i, r_3(e)\in E^0_j \}$,

\item for $i = \{1, 2\}$, there are range and source-preserving bijections $\theta_i : E^1_i \rightarrow  E^2_3 (E^0_i , E^0_i)$, 
where for $i \in  \{1, 2\}$,  $E^2_3(E^0_i, E^0_i) := \{\alpha \in E^2_3 \mid s_3(\alpha) \in E^0_i, r_3(\alpha) \in E^0_i \}.$
\end{enumerate}
Then we say that $E_1$ and $E_2$ are {\it elementary strong shift equivalent} $(E_1 \sim_{ES} E_2)$ {\it via} $E_3$.
\end{deff}

One can prove that $E_1 \sim_{ES} E_2$ via a graph $E_3$ if and only if $A_{E_1}  \sim_{ES} A_{E_2}$ (see~\cite[Proposition~3.10]{ashton}).  

The equivalence relation $\sim_S$ on directed graphs generated by elementary strong shift equivalence is called {\it strong shift equivalence}. In~\cite[Theorem~5.2]{bates02} it was shown that if two row-finite graphs are strongly shift equivalent then their associated
graph $C^*$-algebras are strongly Morita equivalent. We establish a similar statement in the setting of graded Leavitt path algebras. 

In the following by $(1/2) \mathbb Z$ we denote the cyclic subgroup of rational numbers $\mathbb Q$ generated by $1/2$. 

\begin{theorem}\label{cafejen1}
Let $E_1$ and $E_2$ be graphs with no sinks which are elementary strong shift equivalent via the graph $E_3$. 
Then $\LL(E_1)$ and $\LL(E_2)$ are $(1/2) \mathbb Z$-graded Morita equivalent. 
\end{theorem}
\begin{proof}
First observe that if $E_1\sim_{ES} E_2$ via $E_3$ and $E_1$ and $E_2$ have no sinks, then $E_3$ does not have sinks either. 

We first give a $(1/2)\mathbb Z$-graded structure to the Leavitt path algebras $\LL(E_i)$, $i=1,2,3$. We then show that $\LL(E_1)$ and $\LL(E_2)$ are $(1/2)\mathbb Z$-graded Morita equivalent to $\LL(E_3)$. This implies that $\LL(E_1)$ is $(1/2)\mathbb Z$-graded Morita equivalent to $\LL(E_2)$. 

Let $\Gamma=(1/2)\mathbb Z$ and set $\deg(v)=0$, for $v\in E_i^0$, $\deg(\alpha)=1$ and $\deg(\alpha^*)=-1$ for $\alpha \in E_i^1$, where $i=1,2$.  By \S\ref{poidget}, we then obtain a natural $(1/2)\mathbb Z$-grading on $\LL(E_1)$ and $\LL(E_2)$ with support $\mathbb Z$. 

Furthermore, set $\deg(v)=0$, for $v\in E_3^0$, $\deg(\alpha)=1/2$ and $\deg(\alpha^*)=-1/2$ for $\alpha \in E_3^1$. Thus we get a natural $(1/2)\mathbb Z$-grading on $\LL(E_3)$, with support $(1/2)\mathbb Z$.

Next we construct a $(1/2)\mathbb Z$-graded ring homomorphism $\phi: \LL(E_1)\rightarrow \LL(E_3)$ such that $\phi(\LL(E_1))=p\LL(E_3)p$, where $p=\sum_{v\in E_1^0} v \in \LL(E_3)$. Define $\phi(v)=v$ for $v\in E_1^0$ and $\phi(l)=ef$, where $l\in E_1^1$ (and is of degree 1)  and $e f$ is a path of length $2$ in $E_3$  (and is of degree $1$) assigned uniquely to $l$ via $\theta_1$ in Definition~\ref{ttjeaong}(3). We check that the set 
$\{\phi(v),\phi(l)\,  | \,  v\in E_1^0, l\in E_1^1\}$ is an $E$-family in $\LL(E_3)$. 

Let $v\in E_1^0$ and $\phi(v)=v \in E_3^0$. Then 
\begin{align*}
v&=\sum_{\{e\in E_3^1 \mid s_3(e)=v\}}ee^*\\
&=\sum_{\{ef\in E_3^2(E_1^0,E_1^0) \mid s_3(ef)=v\}}ef(ef)^*\\
&=\sum_{\{l\in E_1^1 \mid s_1(l)=v\}}\phi(l)\phi(l)^*
\end{align*}
where the last equality uses condition (c) in Definition~\ref{ttjeaong}. On the other hand for $l \in E_1^1$, 
\[
\phi(l)^*\phi(l)=(ef)^*(ef)=f^*e^*ef=f^*f=r(f)=\phi(r(l)),   
\] since the map $\theta_1$ in the Definition~\ref{ttjeaong} is range-preserving bijection. The rest of the relations of an $E$-family is easily check and thus there is a map $\phi: \LL(E_1)\rightarrow \LL(E_3)$. 

Clearly $\phi(\LL(E_1)) \subseteq p\LL(E_3)p$. Since  $p\LL(E_3)p$ is generated by elements of the form $p \alpha \beta^* p$, where $\alpha,\beta \in E_3^*$, we check that $p \alpha \beta^* p \in \phi(\LL(E_1))$. 
 Clearly either $p \alpha \beta^* p=\alpha\beta^*$ if $s_3(\alpha),s_3(\beta)\in E_1^0$ or $p \alpha \beta^* p=0$ otherwise. 
Furthermore, for $\alpha \beta^*\not =0$, we should have $r_3(\alpha)=r_3(\beta)$. Thus for an element $p \alpha \beta^* p\not =0$, we have either $s_3(\alpha),s_3(\beta), r_3(\alpha), r_3(\beta)\in E_1^0$ or $s_3(\alpha),s_3(\beta)\in E_1^0$ and $r_3(\alpha), r_3(\beta)\in E_2^0$. In the first case, by the construction of $E_3$, the length of $\alpha$ and $\beta$ are even and since $\theta_1$ is bijective, $\alpha$ and $\beta$ are in the images of the map $\phi$. In the second case, since $E_3$ has no sinks, we have
\[\alpha\beta^*=\sum_{\{e\in E_3^1 \mid s(e)=r(\alpha)\}} \alpha e e^* \beta^* =\sum_{\{e\in E_3^1 \mid s(e)=r(\alpha)\}} \alpha e  (\beta e)^*.\]
Now, again by the construction of $E_3$,  for each $e$, $s_3(\alpha e),s_3(\beta e), r_3(\alpha e), r_3(\beta e)\in E_1^0$ and by the first case, they are in the image of the map $\phi$. So $\phi(\LL(E_1))=p\LL(E_3)p$. 

Since $\phi(v)=v \not = 0$ (see~\cite[Lemma~1.5]{goodearl}) by the graded uniqueness theorem~\cite[Theorem~4.8]{tomforde} (which is valid for $(1/2)\mathbb Z$-graded Leavitt path algebras as well)  $\phi$ is injective. Thus  $\LL(E_1)\cong_{\gr} p\LL(E_3)p$. This immediately implies $\LL(E_1)$ is  $(1/2)\mathbb Z$-graded Morita equivalent to $p\LL(E_3)p$. On the other hand, the (graded) ideal generated by $p=\sum_{v\in E_1^0} v$ in $\LL(E_3)$ coincides with the (graded) ideal generated by $\{v \mid v\in E_1^0\}$. But the smallest hereditary and saturated subset of $E_3^0$ containing $E_1^0$ is $E_3^0$. Thus by~\cite[Theorem~5.3]{amp} the ideal generated by $E_1^0$ is $\LL(E_3)$. This shows $p$ is a full homogeneous idempotent in $\LL(E_3)$. 
Thus $p\LL(E_3)p$ is graded Morita equivalence to $\LL(E_3)$ (see Example~\ref{idempogr}). Putting these together we get $\LL(E_1) \approx_{\gr} \LL(E_3)$. In a similar manner $\LL(E_2) \approx_{\gr} \LL(E_3)$. Thus $\LL(E_1) \approx_{\gr} \LL(E_2)$ as $(1/2)\mathbb Z$-graded rings. 
 This finishes the proof. 
\end{proof}

\begin{remark}
Since throughout this paper we are interested in purely infinite simple Leavitt path algebras, which are (graded) simple, we could have specialised Theorem~\ref{cafejen1} to graphs associated to these algebras, and instead of using~\cite[Lemma~1.5]{goodearl} and~\cite[Theorem~5.3]{amp} in the proof, we could only use the simplicity of the algebras.

In~\cite{flowa}, it was shown that if $E$ is a finite graph with no sources and sinks such that $\LL(E)$ is simple, then the Leavitt path algebras obtained from the in-splittings of $E$ (\cite[Proposition~1.11]{flowa}) or out-splittings of $E$ (\cite[Proposition~1.14]{flowa}) are Morita equivalent to $\LL(E)$.  For a graph $E$ with no sinks, it was shown that an in-splitting of a graph $E$~\cite[Proposition~6.2]{batespask} or an out-splitting of $E$~\cite[Proposition~6.3]{batespask} is elementary shift equivalent to the graph $E$. 
Now using Theorem~\ref{cafejen1} and Theorem~\ref{grmorim11}, we can obtain these results in a unified manner. 
\end{remark}

\end{document}